\documentclass[12pt]{article}
\usepackage{graphicx}
\usepackage[dvips]{epsfig}
\usepackage{subfigure}
\usepackage[a4paper]{geometry}
\usepackage{tikz-cd}
\usepackage{amsmath,amssymb,amsthm,amsfonts,hyperref,amscd}
\usepackage{times}
\usepackage{setspace}
\usepackage{verbatim}
\usepackage{cancel}
\usepackage{mathtools}
\usepackage[toc,page]{appendix}
\usepackage{multirow}

\newtheorem{thm}{Theorem}[section]
\newtheorem{defn}[thm]{Definition}
\newtheorem{lem}[thm]{Lemma}

\newtheorem{cor}[thm]{Corollary}
\newtheorem{rmk}[thm]{Remark}
\newtheorem{eg}[thm]{Example}
\theoremstyle{plain}
\theoremstyle{definition}
\counterwithin*{equation}{section}

\newcommand{\F}{\mathcal{F}}

\newcommand{\Db}{\Delta_B}

\newcommand{\nD}{\nabla}

\newcommand{\addresses}{\bigskip\footnotesize

J. Moon, \textsc{Department of Mathematics, Chung-Ang University, 84 HeukSeok-ro DongJak-gu, Seoul 06974, Republic of Korea.} \par\nopagebreak
\textit{E-mail address:}\texttt{dsfish999@cau.ac.kr} }

\title{On vacuum static foliations}
\author{Jungwoo Moon}
\date{}

\begin{document}

\maketitle

\begin{abstract}
We study the compact vacuum static foliation, a Riemannian foliation which generalizes the compact vacuum static space. In particular, we investigate the relations between the tautness of a compact vacuum static foliation and the bound of transverse curvatures.
\end{abstract}

\section{Introduction}



Let $(M,g)$ be a compact Riemannian manifold and $\mathcal{M}$ be the set of all smooth Riemannian metric on $M$. A central topic in differential geometry is the study of the scalar curvature map $\displaystyle{S_g:\mathcal{M} \xrightarrow{} C^{\infty}(M)}$, which sends $g$ to its scalar curvature $S_g$. In particular, Kazdan and Warner emphasized the importance of the surjectivity of $S_g$, which is now known as the \textit{prescribed scalar curvature problem}\cite{KW1, KW2}. To study the local behavior of the prescribed scalar curvature problem, let us consider the linearization of $S_g$:
\begin{equation}
   \dfrac{d}{dt}S_{g}=\mathcal{L}(h)=\text{div}^2h-(\Delta \text{tr}h)-g(h,Ric),
\end{equation} where $Ric$ is the Ricci curvature, $h$ is the variation of the metric $g$ and $\Delta$ is the negative definite Laplacian. If $L$ is not surjective, then $S_g$ fails to be a submersion at $g$ and is difficult to investigate its local behavior (cf. \cite{Kob}). By ellipticity, the surjectivity of $L$ is equivalent to the injectivity of its formal adjoint $\mathcal{L}^*$ \cite{FM}, which is given by 
\begin{equation}
  \mathcal{L}^*(f)=Ddf-(\Delta f)g-fRic,
\end{equation} where $D$ is the Riemannian connection and $Ddf$ is the Hessian of $f$. Thus, a compact Riemannian manifold $(M,g,f)$ with a nontrivial $f \in \ker \mathcal{L}^*$ is essential for understanding the local behavior of $S_g$ and such a compact $(M,g,f)$ is called a \textit{vacuum static space}.

Due to their importance, a compact vacuum static space has been extensively studied \cite{HY,KS,Kob,Laf}. In particular, Fischer and Marsden proved that the scalar curvature $S$ of a compact, nontrivial vacuum static space $(M,g,f)$ should be positive constant \cite{FM}, which was also independently proved by J. P. Bourguignon \cite{Bou}. Motivated by the work of Fischer–Marsden, we study the linearization of the scalar curvature map for Riemannian foliations to understand the prescribed scalar curvature problem on transverse geometry. Few results are known for this problem, apart from the partial results in transverse Yamabe problem \cite{WZ}.

To this end, we consider a compact Riemannian manifold $M$ with a Riemannian foliation $\F$, denoted $(M,\F,g)$ and hereafter called a \textit{compact Riemannian foliation}. Here, a foliation $\F$ is Riemannian if the Riemannian metric $g_Q$ induced on the normal bundle $Q$ of $\F$ satisfies the Lie derivative $L_Ug_Q=0$ for any vector field $U$ tangent to $\F$ (i.e. $g_Q$ is \textit{basic}). We say $\F$ of $(M,\F,g)$ is \textit{taut} if there exists a metric $g'$ on $M$ for which all leaves of $(M,\F,g')$ are minimal. Tautness of $(M,\F,g)$ is one of the central properties of a Riemannian foliation, as shown in \cite{AL, HR, Ton}.
 
Throughout this paper, the criterion of Alvarez-L\'opez is always assumed without loss of generity (cf. Section 2). That is, for the basic component $\kappa_b$ of the mean curvature $\kappa$, $\delta_B\kappa_b=0$, where $\delta_B$ is the formal adjoint of the transverse Riemannian connection $\nD$ \cite{AL}. Consequently, it suffices to show that $\kappa_b=0$ to justify the tautness of a compact $(M,\F,g)$.

In this paper, we investigate the map $\mathcal{M}^Q \xrightarrow{} C^{\infty}(M)$ that send $g_Q$ to its transverse scalar curvature $S^Q$, where $\mathcal{M}^Q$ is the set of the transverse metric $g_Q$ of a compact $(M,\F,g)$. In particular, we are interested in the linearization $\mathcal{L}_Q(h^Q)$ of $S^Q$, which is given by (\cite{Lin})
\begin{equation}
    \mathcal{L}_Q(h^Q)= \text{div}_Q^2\left(h^Q\right)-\text{div}_Q\circ d\left(\text{tr}_Q\left(h^Q\right)\right)-g_Q\left(h^Q,Ric^Q\right),
\end{equation} where $h^Q$ is the linearization of $g_Q$, $Ric^Q$ is the transverse Ricci curvature, and $\text{div}_Q\circ d$ is the negative definite transverse Laplacian.

Since the Riemannian and foliated settings differ only in the divergence operators, the adjoint $\mathcal{L}_Q^*$ of $\mathcal{L}_Q(h^Q)$ on a compact $(M,\F,g)$ plays the same role as $L^*$ does in (1.2). Here, $\mathcal{L}_Q^*$ is given by
\begin{equation}
  \mathcal{L}_Q^*(f)=\nD_{tr} df+(\Delta_B f)g_Q-f(Ric^Q+T_{\kappa}),
\end{equation} where $\displaystyle{T_{\kappa}=\nD_{tr}\kappa_b-\kappa_b\otimes\kappa_b}$, $\nD_{tr}df$ is the transverse Hessian and $\Delta_B f=\delta_B(df)$ is the basic Laplacian of $f$ (cf. Lemma 3.2). The following theorem describes the behavior of $\mathcal{L}_Q(h^Q)$ when it is not surjective, and is an analogue for compact Riemannian foliations of the Fischer-Marsden-Bourguignon theorem \cite{Bou,FM} (cf. Theorem 4.8).
\begin{thm} Let $(M,\F,g)$ be a compact Riemannian foliation. If there is a basic nontrivial function $f$ in $\ker \mathcal{L}_Q^*$ for a given transverse metric $g_Q,$ then $g_Q$ satisfies one of the following.
\begin{enumerate}
    \item $S^Q$ is an eigenvalue of $(q-1)\Delta_B$ if $\kappa_b=0.$
    \item $Ric^Q=0$ and $\ker \mathcal{L}_Q^*=\mathbb{R}$. i.e. $f$ is a real constant.
    \item $dS^Q \neq 0$ if $\kappa_b\neq 0$.
\end{enumerate}
\end{thm}
 Compared with \cite{FM}, it should be remarked that two-codimensional compact Riemannian foliations are not treated here, since counterexamples may exist owing to the lack of a uniformization theorem in this setting (cf. \cite{Heb, WZ}). Also, it is also noteworthy that the third condition of Theorem 1.1 holds if and only if $(M,\F,g)$ is nontaut (cf. Corollary 4.7), which is another difference from the result of Fischer-Marsden-Bourguignon. Consequently, we obtain the following direct corollary of Theorem 1.1.

\begin{cor} Let $(M,\F,g)$ be a compact Riemannian foliation consists of minimal leaves. If there is a basic nonconstant function $f$ in $\ker \mathcal{L}_Q^*$, then $S^Q$ should be positive constant.
\end{cor}

Since a Riemannian foliation with point leaves is clearly minimal, Corollary 1.2 covers the case of vacuum static spaces. For this reason, a compact Riemannian foliation $(M,\F,g,f)$ with a nontrivial basic function $f \in \ker(\mathcal{L}_Q^*)$ is hereafter called a compact \textit{vacuum static foliation}. 

We also note that the transverse vacuum static equation in a naive form is
\begin{equation}
   \nD_{tr} df-(\text{div}_Q\circ df)g_Q-fRic^Q=0,
\end{equation} which coincides with (1.4) if a compact $(M,\F,g,f)$ is taut. otherwise, (1.5) does not coincide with by the transverse divergence formula (cf. (2.4)). Thus, the main difference between (1.4) and (1.5) lies in the curvature terms $Ric^Q+T_{\kappa}$ and $Ric^Q$, which are distinguished on a compact nontaut Riemannian foliation, in particular, on a nontaut vacuum static foliation $(M,\F,g,f)$.

Therefore, in order to study compact nontaut vacuum static foliations, we call the Ricci-type curvature $\displaystyle{Ric^Q+T_{\kappa}}$ the \textit{adjusted Ricci curvature}. We then obtain the following result.

\begin{thm}
 Let $(M,\F,g,f)$ be a compact vacuum static foliation. If $\displaystyle{Ric^Q+T_{\kappa}\geq 0}$, then the given foliation is taut.
\end{thm} 

Note that Theorem 1.3 is also extended to an arbitrary compact $(M,\F,g)$, since the adjusted Ricci curvature $Ric^Q+T_{\kappa}$ is nonnegative if and only if $Ric^Q$ is nonnegative on a compact $(M,\F,g)$ (cf. Theorem 5.4). Hence, the previous partial results on a nonnegative $Ric^Q$ are also covered (cf. \cite{HJM,HR}).

On the other hand, a tautness property with respect to the upper bound of $S^Q$ is also obtained, exclusively on a compact $(M,\F,g,f)$. In particular, the upper bound of $Ric^Q+T_{\kappa}$ on a compact $(M,\F,g,f)$ for tautness is determined since its transverse trace is also identified with $S^Q$. To justify, let us consider a positive constant $\lambda^Q$ given by \begin{equation}
    \lambda^Q\int_M f^2\leq \int_M |df|^2,
\end{equation} for any nonconstant $f\in \ker \mathcal{L}_Q^*$ (cf. (5.11)). Then $\lambda^Q$ is well-defined, as it is a Rayleigh quotient of $\Delta_B$ on $\ker \mathcal{L}_Q^*$ (\cite{Cha}).
Then we could establish the following gap property on a compact vacuum static foliation (cf. Theorem 5.7).
\begin{thm}
    Let $(M,\F,g,f)$ be a compact vacuum static foliation. If \begin{equation}
        S^Q \leq (q-1)\lambda^Q,
    \end{equation} then the given foliation is taut. In particular, $S^Q$ is a constant.
\end{thm} Hence, some well-known compact nontaut Riemannian foliation examples cannot be deemed as examples of vacuum static foliations (cf. Section 6). Therefore, our further research for the vacuum static foliation would be the existence and rigidity of nontrivial, nontaut examples of compact vacuum static foliations.

The content of this paper is structured as follows. In Section 2, we recall the preliminaries of Riemannian foliations. In Section 3, we would derive the adjoint operator of the linearization of transverse scalar curvature. In Section 4, Theorem 1.1 is proved. In section 5, Theorem 1.3 is proven, and so is Theorem 1.4. In the last section, we calculate some nontaut examples of Riemannian foliations to determine whether they are the solution of transverse vacuum static space.

The author is grateful to Professor Seungsu Hwang and Professor Seoung Dal Jung for their helpful discussions to improve this work.

\section{Preliminaries}

Let $(M,\F,g)$ be a complete Riemannian manifold with a $q$-codimensional foliation $\F$ with the tangent distribution $T\F$ and let the vector bundle $Q$ be the vector bundle linear map $\pi: TM \xrightarrow{} Q$ with $\ker\pi=Q$. In particular, we denote $U$ a vector field on $T\F$ and vector fields $X,Y,Z$ on $Q$ whose Lie derivatives with respect to any $U$ vanish. Also, let $\tau$ be the mean curvature of $(M,\F,g)$

Since the given foliated manifold is a Riemannian manifold, we have the induced metric $g_Q$ on $Q$ defined by $g_Q=\sigma^*g$, where $\sigma:Q \xrightarrow{} TM$ is the vector bundle embedding. Therefore, we identify $Q$ and $\sigma(Q)$ throughout this paper. Also, note that the metric $g$ is bundle-like. That is, $g$ is decomposed by $g_{T\F}+g_Q$, where $g_{T\F}$ is the remaining part of the metric $g$ on $T\F$.

Thus, a \textit{Riemannian foliation} $(M,\F,g)$ is a foliation on a complete manifold satisfying $L_Ug_Q=0$ for any $U$. That is, a Lie derivative with respect to $U$ on $g_Q$ always vanishes. Also, we define the \textit{transverse Riemannian connection} $\nD$ on $(Q,g_Q)$ of $(M,\F,g)$ as follows \cite{Ton}.

\begin{equation}
\begin{cases}
\nD_UY =\pi[U,Y] \\
\nD_XY =\pi(D_XY),
\end{cases}
\end{equation} where $D$ is the ordinary Riemannian connection of $(M,g)$. In particular, we denote $\nD_{tr}$ if we only consider the transverse Riemannian connection $\nD_X$ for an arbitrary vector field $X$ on $Q$.

In light of the definition of curvature tensor, we define the transverse Riemann curvature tensor by the following.
\begin{equation}
R^Q(X,Y)Z=\nD_Y\nD_XZ-\nD_X\nD_YZ+\nD_{[X,Y]}Z.
\end{equation} In particular, $R^Q$ is the $4$-tensor defined as follows.
\begin{equation}
R^Q(X,Y,Z,Z')=g_Q(R^Q(X,Y)Z,Z').
\end{equation}
Clearly, $R^Q$ is a basic tensor as the transverse metric $g_Q$ is Riemannian. In other words, both the interior product $i_UR^Q$ and the Lie derivative $L_UR^Q$ vanish for every $U$ \cite{Ton}. Likewise, we define the basic symmetric $2$-tensor called \textit{transverse Ricci curvature} $Ric^Q$ (respectively, the basic function called \textit{transverse scalar curvature} $S^Q$) by taking the transverse trace $\text{tr}_Q$ on $R^Q$(respectively, basic symmetric $2$-tensor $Ric^Q$).

Now let us consider the transverse divergence $\displaystyle{\text{div}_Q=\sum_{i=1}^qi_{\mathbf{e}_i}\nD_{\mathbf{e}_i}}$ on $Q$, where $\{\mathbf{e}_i\}$ is a transverse moving frame on $Q$ and $i_{\mathbf{e}_i}$ is the interior product of tensors with respect to $\mathbf{e}_i$. It is well-known that $\text{div}_Q$ may not be the formal adjoint operator of the transverse Riemannian connection $\nD$ \cite{Ton}. Instead, we use the formal adjoint operator $\delta_B$ of $\nD$, called the \textit{basic divergence}, defined as follows.
\begin{equation}
\delta_B=-\text{div}_Q+i_{\tau_b},
\end{equation} where $\tau_b$ is the basic component of mean curvature vector field of $\F$ whose corresponding $1$-form $\kappa_b$ is a basic, closed $1$-form \cite{AL,Ton}. Thus, $\nD_{tr}\kappa_b$ is also a basic symmetric $2$-tensor. Note that a Riemannian foliation is \textit{taut}(or, minimalizable) if its basic component of mean curvature $\kappa_b$ is exact \cite{Ton}. Otherwise, we say a foliation is \textit{nontaut}. Therefore, a foliation on $(M,\F,g)$ is taut if and only if $\displaystyle{\delta_B=\text{div}_Q}$ without loss of generality \cite{AL, Ton}. By the work of Alvarez-L\'opez again, we may always adjust that the basic component of mean curvature $1$-form $\kappa_b$ is basic coclosed(i.e. $\delta_B\kappa_b=0$) under the fixed transverse metric\cite{AL}. For this reason, we always assume the aforementioned \textit{Alvarez-L\'opez criterion}. i.e. $\delta_B\kappa_b=0$ throughout this paper.

For later use, we review the following facts on a compact Riemannian foliation. 

\begin{lem}\cite{HJM,Lin} Let $(M,\F,g)$ be a compact Riemannian foliation satisfying $h_{\F}=0$, where $h_{\F}$ is the linearization of $g_{T\F}$. Then the linearization of $S^Q$ is derived as follows.
\begin{equation}
   \mathcal{L}_Q= \text{div}_Q^2\left(h^Q\right)-\text{div}_Q\circ d\left(\text{tr}_Q\left(h^Q\right)\right)-g_Q\left(h^Q,Ric^Q\right).
\end{equation}
\end{lem}

\begin{lem}\cite{Rum}
Let $\chi_{\F}$ be the characteristic form of a compact $(M,\F,g)$. That is, the volume form of $T\F$. Then the following formula holds.
\begin{equation}
    d\chi_{\F}=-\kappa_b\wedge\chi_{\F}+\phi_0,
\end{equation} where $\phi_0$ is a $(p+1)$-form satisfying $\phi_0(U_1,...,U_p,\cdot)=0$ for any linearly independent $p$ vector fields $U_i,...,U_p$ tangent to $T\F$.  
\end{lem}

We finish this section to recall the transverse analogies of the Weitzenb\"ock formula below for later use.

\begin{thm}\cite{Jun}
On a compact $(M,\F,g)$, the following equation holds for any basic $1$-form $\eta$.
\begin{equation}
\Db\eta=Ric^Q(\eta)+\nD_{tr}^*\nD_{tr}\eta+\mathcal{A}_{\tau_b}\eta,
\end{equation} where $\nD_{tr}^*=\delta_B$ is the formal adjoint of the transverse Riemannian connection, $\mathcal{A}_{\tau_b}$ is a differential operator given by $\mathcal{A}_{\tau_b}=L_{\tau_b}-\nD_{\tau_b},$ $Ric^Q$ is the transverse Ricci curvature identified with an automorphism on $Q$, and $\Db$ is the basic Laplacian given by $\displaystyle{\Db=\delta_Bd+d\delta_B}$.
Moreover, the following is induced from the above equation.
\begin{equation}
-\dfrac{1}{2}\Db|\eta|^2=Ric^Q(\eta^{\sharp},\eta^{\sharp})+|\nD_{tr}\eta|^2+g_Q((\mathcal{A}_{\tau_b}\eta)^{\sharp},\eta^{\sharp})-g_Q(\eta^{\sharp},(\Db\eta)^{\sharp}).
\end{equation}
\end{thm}

Similarly, Habib and Richardson have established the following version of Weitzenb\"ock formula, which is useful for the later content \cite{HR}.
\begin{equation}
    \tilde{\Delta}\eta=Ric^Q(\eta)+\nD_{tr}^*\nD_{tr}\eta+\dfrac{1}{4}|\kappa_b|^2\eta,
\end{equation} where $\tilde{\Delta}$ is an elliptic differential operator defined as follows. 

\begin{equation}
    \tilde{\Delta}=(d-\dfrac{1}{2}\kappa_b\wedge)\circ(\delta_B-\dfrac{1}{2}i_{\kappa_b})+(\delta_B-\dfrac{1}{2}i_{\kappa_b})\circ(d-\dfrac{1}{2}\kappa_b\wedge).
\end{equation}

\section{The linearization of transverse scalar curvature map}

Motivated by D. Lin's linearization on transverse scalar curvature in Lemma 2.1, we derive the adjoint mapping of the aforementioned linearization on a compact $(M,\F,g)$.

Since we are not interested in the linearization of leaf directions, let us assume that the linearization of $g_{T\F}$ vanishes (i.e. $h_{T\F}=0$) throughout this paper. Then for a basic function $f$, we have
\begin{equation}
\int_M f\mathcal{L}_Q(h^Q)=\int_M f(\text{div}_Q^2h^Q-\text{div}_Q\circ d(\text{tr}_Qh^Q)-g_Q(h^Q,Ric^Q)).
\end{equation} Applying (2.4), the following formulae for the divergence terms of the above equation are derived. 
\begin{equation}
\begin{split}
     \int_M f\text{div}_Q^2\left(h^Q\right)&=\int_M g_Q(-df+f\kappa_b,\text{div}_Qh^Q)  \\
 &=\int_M g_Q((-df+f\kappa_b )\otimes\kappa_b+\nD_{tr} df -\nD_{tr} (f\kappa_b),h^Q)
\end{split}
\end{equation}
\begin{equation}
\begin{split}
    -\int_M f \text{div}_Q\circ d\left(\text{tr}_Q\left(h^Q\right)\right)&=\int_M \Delta_Bf\text{tr}_Q\left(h^Q\right)-\int_M g_Q(f\kappa_b,d\text{tr}_Q\left(h^Q\right))\\
    &=\int_M g_Q(\Delta_Bfg_Q,h^Q)+\int_M g_Q(df,\kappa_b)\text{tr}_Q\left(h^Q\right),
\end{split} 
\end{equation} since $\delta_B\kappa_b=0$.
Thus, we obtain 
\begin{equation}
\begin{split}
\int_M g_Q(\mathcal{L}_Q(h^Q),f)=&\int_M g_Q(\nD_{tr} df+\Delta_Bfg_Q-fRic^Q-f(\nD_{tr}\kappa_b-\kappa_b\otimes\kappa_b),h^Q)\\
&+\int_M g_Q(df,\kappa_b)\text{tr}_Q\left(h^Q\right)-2\int_M g_Q(df\otimes\kappa_b,h^Q).
\end{split}
\end{equation}
To apply on the above equation, we check the following lemma.
\begin{lem} On a compact $(M,\F,g)$,
    \begin{equation}
        \int_M g_Q(df,\kappa_b)\text{tr}_Q\left(h^Q\right)=2\int_M g_Q(df\otimes\kappa_b,h^Q).
    \end{equation}
\end{lem}
\begin{proof}
    Note that the linearization of the Riemannian volume form is given as follows \cite{CLN}. \begin{equation}
        \dfrac{d}{dt}dV=\dfrac{1}{2}\text{tr}\left(\dfrac{d}{dt}(g_Q+g_{T\F})\right)dV=\dfrac{1}{2}tr_Q\left(h^Q\right)dV,
    \end{equation} as we have assumed that $h_{T\F}=0.$
    Thus, \begin{equation}
        \int_M g_Q(df,\kappa_b)\text{tr}_Q\left(h^Q\right)=2\int_M g_Q(df,\kappa_b)\dfrac{d}{dt}dV.
    \end{equation}
    On the other hand, applying the product rule with respect to $\dfrac{d}{dt},$ we obtain
    \begin{equation}
    \begin{split}
     \int_M g_Q(df,\kappa_b)\dfrac{d}{dt}dV&=\dfrac{d}{dt}\left(\int_M g_Q(df,\kappa_b)\right)+\int_M g_Q(h^Q,df\otimes \kappa_b)dV\\&-\int_M g_Q(df, \dfrac{d}{dt}\kappa_b)-\int_M g_Q(\dfrac{d}{dt}(df), \kappa_b).
    \end{split}
    \end{equation} Note that by Rummler's formula, 
    \begin{equation}
        -\dfrac{d}{dt}(\kappa_b\wedge\chi_{\F})=\dfrac{d}{dt}d\chi_{\F}-\dfrac{d}{dt}\phi_0=-\dfrac{d}{dt}\phi_0
    \end{equation} for some $(p+1)$-form $\phi_0,$ and hence 
     \begin{equation}
        \dfrac{d}{dt}\kappa_b=0
    \end{equation} is obtained, as $\dfrac{d}{dt}g_{T\F}=0$ is assumed.   
On the other hand, although $g_{T\F}$ may change, $\delta_B\kappa_b=0$ could be also assumed without loss of generality. This implies \begin{equation}
     \int_M g_Q(df,\kappa_b)\dfrac{d}{dt}dV=\int_M g_Q(h^Q,df\otimes \kappa_b)dV,
    \end{equation} as $\dfrac{d}{dt}$ and the exterior derivative $d$ of $f$ should commute by the definition of $d$.
    Therefore, we have proven the assertion.
\end{proof}
For these reasons, (3.1) is deduced as follows. 
\begin{equation}
    \int_M f\mathcal{L}_Q(h^Q)=\int_M g_Q(\nD_{tr} df+\Delta_Bfg_Q-f(Ric^Q+\nD_{tr}\kappa_b-\kappa_b\otimes\kappa_b),h^Q).
\end{equation}
That is, the definition of the adjoint of $\mathcal{L}_Q$ should be given as follows.
\begin{lem} Let $(M,\F,g)$ be a compact Riemannian foliation. Then the adjoint of the linearized scalar curvature is derived as follows. 
    \begin{equation}
   \mathcal{L}_Q^*(f)=\nD_{tr} df+\Delta_Bfg_Q-f(Ric^Q+\nD_{tr}\kappa_b-\kappa_b\otimes\kappa_b).
\end{equation} 
\end{lem}

In particular, we call the equation\begin{equation}
   \mathcal{L}_Q^*(f)=0
\end{equation} \textit{vacuum static foliation equation} and a compact $(M,\F,g,f)$ with a nontrivial basic function $f \in \ker \mathcal{L}_Q^*$ is called a compact \textit{vacuum static foliation}. 

Concerning the above definition of $\mathcal{L}_Q^*$, we define the \textit{Adjusted Ricci curvature} as follows.
\begin{defn}
An adjusted Ricci curvature $A_{\kappa}$ is a basic $2$-tensor defined by
\begin{equation}
 A_{\kappa}=Ric^Q+T_{\kappa},
\end{equation} where $\displaystyle{T_{\kappa}= \nD_{tr}\kappa_b-\kappa_b\otimes\kappa_b,}$ which is called symmetric tautness tensor hereafter.
\end{defn}

Note that the criterion of Alvarez-L\'opez, we assume that the basic component of mean curvature $\kappa_b$ satisfies $\delta_B\kappa_b=0$ and $d\kappa_b=0$. Then we have $\displaystyle{\text{tr}_Q(T_{\kappa})=0}$ by direct calculation. Therefore, we always consider that $T_{\kappa}$ is \textit{transverse trace-free} by the assumption $\delta_B\kappa_b= 0$ throughout this paper.

\section{Transverse vacuum static spaces}
Contrary to the linearized scalar curvature of a Riemannian manifold, the map $\mathcal{L}_Q^*$ has the extra term $T_{\kappa}$. Hence, we should investigate some geometric properties of the symmetric tautness tensor and adjusted Ricci curvature, to deal with the vacuum static foliation equation.

For later use, the following lemmas are about to check.
\begin{lem} On a compact $(M,\F,g)$,
    $\nD_{tr}\kappa_b=0$ if and only if $\kappa_b=0$.
\end{lem}
\begin{proof}
 To prove, we only need to suppose that $\nD_{tr}\kappa_b=0$. Taking the transverse trace, we have the following.
 \begin{equation}
     0=\text{tr}_Q(\nD_{tr}\kappa_b)=\text{div}_Q\kappa_b=-\delta_B\kappa_b+|\kappa_b|^2=|\kappa_b|^2.
 \end{equation} Hence, $\kappa_b=0$ is obtained.
\end{proof}
Similar to the behavior of $\nD_{tr}\kappa_b,$ we have the following lemma.

\begin{lem}
On a compact $(M,\F,g)$, $\kappa_b=0$ if and only if $T_{\kappa}=0.$
\end{lem}
\begin{proof} We only need to prove $T_{\kappa}=0$ implies $\kappa_b=0$. Since we have  \begin{equation}
\int_M g_Q(\nD_{tr}\kappa_b,\kappa_b\otimes\kappa_b)=\int_M g_Q(\nD_{\kappa_b}\kappa_b,\kappa_b)=\dfrac{1}{2}\int_M \delta_B\kappa_b|\kappa_b|^2=0,
\end{equation} we have \begin{equation}
    \int_Mg_Q(T_{\kappa},T_{\kappa})=\int_M|\nD_{tr}\kappa_b|^2+\int_M|\kappa_b|^4=0.
\end{equation} Hence, $\kappa_b=0$ by positive definiteness and Lemma 3.2.
\end{proof}

\begin{lem}
On a compact $(M,\F,g)$, the following formula holds.
\begin{equation}
-\delta_B(T_{\kappa})=Ric^Q(\kappa_b).
\end{equation}
\end{lem}

\begin{proof}
Applying (2.7), we obtain
\begin{equation}
0=Ric^Q\kappa_b+\nD_{tr}^*\nD_{tr}\kappa_b+L_{\tau_b}\kappa_b-\nD_{\tau_b}\kappa_b.
\end{equation} On the other hand, we also have
\begin{equation}
-\delta_B(\kappa_b\otimes\kappa_b)=\nD_{\tau_b}\kappa_b.
\end{equation}
Adding (4.5) from (4.6), the following is deduced.
\begin{equation}
-\delta_B(T_{\kappa})=Ric^Q\kappa_b+L_{\tau_b}\kappa_b-2\nD_{\tau_b}\kappa_b
\end{equation}
By Cartan's magic formula, the following equation is obtained.
\begin{equation}
L_{\tau_b}=2\nD_{\tau_b}\kappa_b.
\end{equation}
Thus,
\begin{equation}
-\delta_B(T_{\kappa})=Ric^Q(\kappa_b),
\end{equation} is derived, as desired.
\end{proof}

As a corollary of the above lemma, we have the following. 

\begin{lem}
On a compact $(M,\F,g)$, 
\begin{equation}
-\delta_B(A_{\kappa})=\dfrac{1}{2}dS^Q.
\end{equation}
\end{lem}
\begin{proof}
    Since \begin{equation}
        -\delta_BRic^Q=\dfrac{1}{2}dS^Q-Ric^Q({\kappa_b}), 
    \end{equation} the calculation is straightforward.
\end{proof}

On the other hand, we should justify the following lemma for later use.

\begin{lem}
On a compact $(M,\F,g)$, the following formula holds.
    \begin{equation}
        \int_Mg_Q(Ric^Q,T_{\kappa})=0.
    \end{equation}
\end{lem}
\begin{proof}
    Due to the second Bianchi identity with respect to $R^Q,$ it is well-known that 
    \begin{equation}
    \text{div}_QRic^Q=\dfrac{1}{2}dS^Q
    \end{equation} holds \cite{Ton}.
    Thus, we have
      \begin{equation}
        \int_Mg_Q(Ric^Q,T_{\kappa})=-\dfrac{1}{2}\int_M g_Q(dS^Q,\kappa_b).
    \end{equation}
    Since \begin{equation}
        \dfrac{1}{2}\int_M g_Q(dS^Q,\kappa_b)=\dfrac{1}{2}\int_MS^Q\delta_B\kappa_b, 
    \end{equation} we obtain \begin{equation}
        \int_Mg_Q(Ric^Q,T_{\kappa})=0
    \end{equation} as a consequence.
\end{proof}
 
Applying to the basic divergence of the transverse vacuum static equation, we have the following result.
\begin{thm}
    Let $(M,\F,g,f)$ be a compact vacuum static foliation. Then we have
    \begin{equation}
        -(\Delta_Bf-\tau_b(f))\kappa_b=\dfrac{1}{2}fdS^Q.
    \end{equation}
\end{thm}
\begin{proof}
    By definition, we have
    \begin{equation}
        0=\nD_{tr}^*\nD_{tr}df+\delta_B(\Delta_Bfg_Q)+A_{\kappa_b}(df)+\dfrac{1}{2}fdS^Q.
    \end{equation}
    Applying the transverse Weitzenb\"ock formula, we also have
    \begin{equation}
        \nD_{tr}^*\nD_{tr}df-d\Delta_Bf=-Ric^Q(df)-\mathcal{A}_{\tau_b}df=-Ric^Q(df)-dg_Q(df,\kappa_b)+\nD_{\tau_b}df,
    \end{equation} by Cartan's magic formula.
    That is, \begin{equation}
        \nD_{tr}^*\nD_{tr}df-d\Delta_Bf=-Ric^Q(df)-\mathcal{A}_{\tau_b}df=-Ric^Q(df)-\nD_{tr}\kappa_b(df),
    \end{equation} by definition of transverse Levi-Civita connection.
    Thus, \begin{equation}
        0=-(Ric^Q+\nD_{tr}\kappa_b)(df)+\Delta_Bf\kappa_b+A_{\kappa}(df)+\dfrac{1}{2}fdS^Q
    \end{equation} is obtained, and the above is deduced by
    \begin{equation}
        0=\Delta_Bf\kappa_b-g_Q(\kappa_b,df)\kappa_b+\dfrac{1}{2}fdS^Q,
    \end{equation} as desired.
\end{proof}

Contrary to the vacuum static space, the solution of transverse vacuum static space equation may not have a constant transverse scalar curvature. In particular, the following corollary is obtained.
\begin{cor}
     Let $(M,\F,g,f)$ be a compact vacuum static foliation. Then $S^Q$ is constant if and only if $\kappa_b=0$.
\end{cor}
\begin{proof}
    Assume that $dS^Q=0$ and $M$ is connected. Then we have \begin{equation}
        (\Delta_Bf-\tau_b(f))\kappa_b=0.
    \end{equation} Thus, $\kappa_b=0$ or $\Delta_Bf-\tau_b(f)=0$ at each point is obtained. At points satisfying $\kappa_b=0,$ we have \begin{equation}
        q\Delta_Bf=fS^Q
    \end{equation} and \begin{equation}
        (q-1)\Delta_Bf=fS^Q
    \end{equation} holds at points satisfying  $\Delta_Bf-\tau_b(f)=0$. Since $\kappa_b=0$ if and only if $|\kappa_b|^2=0,$ we obtain that $(M,\F,g,f)$ is a union of the zero set of $\Delta_Bf-\tau_b(f)$ and the zero set $\kappa_b$. Hence, the intersection between the aforementioned two zero sets is nonempty and $S^Q$ should be forced to vanish on $(M,\F,g,f)$, since $S^Q$ is constant. As a consequence, $\Delta_Bf=0$ is obtained. Note that by the maximum principle on compact manifolds, $\Delta_Bf=0$ implies that $f$ is a constant. That is, the transverse vacuum static space equation is deduced as follows.
    \begin{equation}
        0=f(A_{\kappa}).
    \end{equation} Thus, $A_{\kappa}=0$ is obtained.
   Therefore, Lemma 4.5 guarantees that \begin{equation}
       \int_M |A_{\kappa}|^2=\int_M |Ric^Q|^2+\int_M |T_{\kappa}|^2=0
   \end{equation} and the assertion is proven by positive definiteness.
   On the other hand, if we assume the converse, then we obtain $dS^Q=0$ straightforwardly.
\end{proof}
Applying Corollary 4.7, the following theorem is proven.

\begin{thm}
Let $(M,\F,g,f)$ be a compact vacuum static foliation. Then one of the following holds. 
\begin{enumerate}
    \item $S^Q$ is an eigenvalue of $(q-1)\Delta_B$.
    \item $(M,\F,g,f)$ is Ricci-flat and $\ker \mathcal{L}_Q^*=\mathbb{R}$. 
    \item $dS^Q \neq 0$ and $(M,\F,g,f)$ is nontaut. 
\end{enumerate}
\end{thm}
\begin{proof}
    We may assume that $(M,\F,g,f)$ is a compact taut vacuum static foliation, as a nontaut case is, de facto, proven in Corollary 4.7. Then, we have 
    \begin{equation}
    fdS^Q=0
    \end{equation} from (4.22). That is, $S^Q$ is constant. Hence the proof is parallel to the proof on a compact vacuum static space case in Besse \cite{Bes}. However, we write down the justification for the sake of completeness. 
    Thus, we have 
    \begin{equation}
      (q-1)\Delta_Bf=fS^Q  
    \end{equation} and if $S^Q>0,$ then it is clearly an eigenvalue of $(q-1)\Delta_B$ by definition. On the other hand, if $S^Q=0,$ then we have
    \begin{equation}
        \Delta_Bf=0
    \end{equation} and the maximum principle on $(M,\F,g,f)$ forces $f$ constant. That is, $\ker \mathcal{L}_Q^*=\mathbb{R}$ is guaranteed and the equation of vacuum static foliation deduces by (4.26), which directly induces $Ric^Q=0$ by (4.27).
\end{proof}

\section{Properties of Adjusted Ricci curvature}
Now we prove some tautness properties on a compact vacuum static foliation, related to the bounds of adjusted Ricci curvature. 

Before we consider the property of nonnegative $A_{\kappa},$ we prove the property of nonnegative transverse curvature, which recovers the result of Habib and Richardson \cite{HR} and the work of S. Hwang, S. D. Jung, and the author \cite{HJM}.
\begin{lem}
A compact $(M,\F,g)$ is necessarily taut if and only if $ Ric^Q(\tau_b,\tau_b)$ is nonnegative.
\end{lem}
\begin{proof}
As  \begin{equation}
-\int_M |\nD_{tr}\kappa_b|^2dV=\int_M Ric^Q(\tau_b,\tau_b)dV
\end{equation} is implied by (3.9) and (4.1), $|\nD_{tr}\kappa_b|^2=0$ if and only if $Ric^Q(\tau_b,\tau_b) \geq 0$ by Lemma 3.2. Therefore, the assertion is justified by Lemma 3.2 once again.
\end{proof}

\begin{rmk}\normalfont
 Applying Habib-Richardson type Weitzenb\"ock formula (2.9), we have the following observation. If we choose a basic-harmonic $1$-form $\eta,$ then we have \begin{equation}
     \int_M g_Q(\tilde{\Delta}\eta,\eta)=\dfrac{1}{4}\int_M|\kappa_b|^2|\eta|^2
 \end{equation} by a straightforward calculation. Hence, we have 
 \begin{equation}
    \int_M Ric^Q(\eta^{\sharp},\eta^{\sharp})+\int_M|\nD_{tr}\eta|^2=0.
 \end{equation} That is, for a nonnegative $Ric^Q$ with $Ric^Q(p)>0$ for some point $p$, $\Db\eta=0$ implies $\eta=0$ and hence any basic closed $1$-form is exact, including $\kappa_b$. 
\end{rmk}

Note that we also have a similar calculation to Lemma 5.1 for the adjusted Ricci curvature. 
\begin{lem}
   On a compact Riemannian foliation $(M,\F,g)$, we have \begin{equation}
        \int_M A_{\kappa}(\tau_b,\tau_b)=-\int_M|T_{\kappa}|^2.
    \end{equation}
\end{lem}
\begin{proof} We may adjust that $\delta_B\kappa_b=0$. Applying the transverse Weitzenb\"ock formula, we obtain the following formula.
    \begin{equation}
        0=\int_M |\nD_{tr} \kappa_b|^2+\int_M Ric^Q(\kappa_b,\kappa_b).
    \end{equation}
    On the other hand, we also have
    \begin{equation}
        \int_M T_{\kappa}(\tau_b,\tau_b)=-\int_M g_Q(\kappa_b\otimes \kappa_b,\kappa_b\otimes \kappa_b).
    \end{equation}
    Combining both (5.5) and (5.6), we have (5.4), as desired. 
\end{proof}

Hence, $\displaystyle{
A_{\kappa} \geq 0}$ implies that $\displaystyle{-\int|T_{\kappa}|^2\geq 0}$, and the positive definiteness guarantees that $T_{\kappa}=0$ by (5.4). Therefore, the following corollary proves Theorem 1.3 directly, as a compact vacuum static foliation with $A_{\kappa}\geq 0,$ should be taut due to the below. 
\begin{thm}
     Let us consider a compact $(M,\F,g)$. Then $Ric^Q \geq 0$ if and only if $A_{\kappa} \geq 0.$ Moreover, both cases imply that $\F$ is taut.
\end{thm}
\begin{proof} Throughout the proof, we may suppose that $\delta_B\kappa_b=0.$
    Assume that $Ric^Q$ is nonnegative. Then by Lemma 5.1, $T_{\kappa}$ necessarily vanishes.
    Conversely, if we assume that $A_{\kappa} \geq 0,$ then Lemma 5.3 forces that $T_{\kappa}=0$, and hence $Ric^Q\geq 0$ is obtained.
\end{proof}
Note that Theorem 5.4 holds for any compact $(M,\F,g)$. Hence, the Ricci curvature term $Ric^Q(\tau_b,\tau_b)$ in Lemma 5.1 could be always replaced with $A_{\kappa}(\tau_b,\tau_b)$.

On the other hand, we may also consider the upper bound of $S^Q$ of a compact $(M,\F,g,f)$ to guarantee its tautness. To observe the upper bound, the following formula should be calculated.
\begin{lem}
    On a compact vacuum static foliation $(M,\F,g,f)$, the following formula holds.
    \begin{equation}
        \dfrac{q-1}{2}\Delta_B(f^2)+\dfrac{1}{2}\tau_b(f^2)+(q-1)|df|^2-S^Qf^2=0.
    \end{equation}
\end{lem}
\begin{proof}
    By definition of vacuum static foliation, we have
    \begin{equation}
        (q-1)\Delta_Bf+g_Q(\kappa_b,df)-fS^Q=0,
    \end{equation} as the above is the transverse trace of (3.14).
    That is, \begin{equation}
        (q-1)f\Delta_Bf+g_Q(\kappa_b,fdf)-f^2S^Q=0
    \end{equation} also holds.
    So the Leibniz rule allows us to obtain
    \begin{equation}
        (q-1)\delta_B(fdf)+(q-1)|df|^2+\dfrac{1}{2}\tau_b(f^2)-S^Qf^2=0,
    \end{equation} as desired.
\end{proof}

Now let us consider a positive constant $\lambda^Q$, which is the Rayleigh quotient on $\ker \mathcal{L}_Q^*$, defined as follows.
\begin{equation}
    0<\lambda^Q=\inf_{\ker \mathcal{L}_Q^*} \dfrac{||df||^2}{||f||^2},
\end{equation} where \begin{equation}
    ||f||^2=\int_M f^2.
\end{equation} Note that by the maximum principle, $\displaystyle{ \dfrac{||df||^2}{||f||^2}}$ is positive if and only if $f$ is nonconstant. Thus, the infimum $\lambda^Q$ should be positive and the following result is derived by Lemma 5.5. 
\begin{lem}
    Let us assume that \begin{equation}
        S^Q \leq (q-1)\lambda^Q
    \end{equation} on a compact $(M,\F,g,f)$. Then for a nonconstant $f$,
    \begin{equation}
     ({S^Q}-(q-1)\lambda^Q)f^2=0,
    \end{equation} 
\end{lem}
\begin{proof}
    By definition of $\lambda^Q,$
    \begin{equation}
        (q-1)\int_M|df|^2=\int_MS^Qf^2 \geq (q-1)\lambda^Q \int_M f^2, 
    \end{equation} unless $f$ is a constant. On the other hand, the following is obtained from the assumption on $S^Q$.
    \begin{equation}
        \int_MS^Qf^2 \leq (q-1)\lambda^Q \int_M f^2.
    \end{equation} As both $f^2$ and $(q-1)\lambda^Q-S^Q$ are positive, it is straightforward that (5.14) holds, as desired.
\end{proof}

Applying Lemma 5.6, the following is proved.
\begin{thm}
     If a compact $(M,\F,g,f)$ satisfies \begin{equation}
        S^Q \leq (q-1)\lambda^Q,
    \end{equation} then it is taut. In particular, $S^Q$ is a constant.
\end{thm}
\begin{proof}
    Since (5.14) holds on a nonconstant function $f$, we obtain that \begin{equation}
        S^Q = (q-1)\lambda^Q
    \end{equation} for any $(M,\F,g,f)$ with nonconstant $f$. Otherwise, $f=0$ and a contradiction occurs. 
    On the other hand, if $f$ is a constant, then Theorem 4.8 implies that $Ric^Q$ vanishes, and therefore it is clearly a taut foliation, by Theorem 5.1. 
    Combining the two cases, the assertion of this theorem is proved.
\end{proof}

Therefore, if a compact $(M,\F,g,f)$ satisfies (5.17), then its transverse scalar curvature $S^Q$ is equal to either $(q-1)\lambda^Q$ or $0$ by Theorem 5.7.

\section{Examples}
In this section, we compute some examples of a compact manifold with a Riemannian foliation, which cannot be a transverse vacuum static space.

First, the symmetric tautness tensor on a $3$-dimensional Carri\`ere torus is derived as follows. 
\begin{eg}\normalfont($3$-dimensional Carri\`ere torus)\cite{Car}
Consider the following hyperbolic torus $T^3_{A}$ defined as follows. Let $A=(a_{ij}) \in SL_2(\mathbb{Z})$ satisfying $\text{tr}(A)>2$ and $\mathbf{e}_1, \mathbf{e}_2$ be its eigenvectors corresponding to its eigenvalues $\rho$ and $\dfrac{1}{\rho}$, respectively. Let us denote $\mathbf{e}_3 \in \mathbb{R}^3$ by the vector field of the image of the cross product with respect to $\mathbf{e}_1$ and $\mathbf{e}_2$, if we identify the eigenvectors of $A$ with two left-invariant vector fields on $\mathbb{R}^3$. That is, $\mathbf{e}_3 = \mathbf{e}_1 \times \mathbf{e}_2$. Let us also assume the following Lie algebra structure on such $\mathbb{R}^3$.
\begin{equation}
\begin{split}
&[\mathbf{e}_1,\mathbf{e}_2]=0\\
&[\mathbf{e}_1,\mathbf{e}_3]=\text{log}\rho\text{ }\mathbf{e}_1\\
&[\mathbf{e}_2,\mathbf{e}_3]=-\text{log}\rho\text{ }\mathbf{e}_2
\end{split}
\end{equation}

Then we have the following induced Riemannian metric on such a Lie group.
\begin{equation}
g=dx_1^2+dx_2^2+dx_3^2,
\end{equation} where $\{dx_i\}$ is the dual basis of the orthonormal basis $\{\mathbf{e}_i\}$.

Also, let us think of the following cocompact relation on $(\mathbb{R}^3,[,])$.
\begin{equation}
\{(x_1,x_2,x_3)\in \mathbb{R}^3|(x_1+\rho^{n_3}+n_1)\mathbf{e}_1+(x_2+\rho^{-n_3}+n_2)\mathbf{e}_2+(x_3+n_3)\mathbf{e}_3, n_1,n_2,n_3 \in \mathbb{Z} \}.
\end{equation}

Then the obtained compact hyperbolic torus has the induced Riemannian foliation, whose transverse space is spanned by $\{\mathbf{e}_2,\mathbf{e}_3\}$ with the transverse metric \begin{equation}
g_Q=dx_2^2+dx_3^2.
\end{equation} This foliation is known as \textit{Carri\`ere's torus}, which is the most famous example of a nontaut Riemannian foliation on a compact manifold(we refer to the proof, for example, \cite{Car, HR}). According to \cite{HJM}, we may calculate that its transverse metric is transverse Einstein, which satisfies \begin{equation}Ric^Q=-(\text{log}\rho)^2g_Q.\end{equation}

Now, let us check whether the symmetric tautness tensor and its related terms are nonzero.
First, we have \begin{equation}\kappa_b=-\text{log}\rho\text{ }dx_3\end{equation} by direct calculation. Hence, we obtain \begin{equation}T_{\kappa}=(\text{log}\rho)^2(dx_2\otimes dx_2-dx_3\otimes dx_3)\end{equation} by direct calculation of transverse Riemannian connection. 

 Therefore, the symmetric tautness tensor on $T^3_A$ do not vanish. That is, we have alternatively justified that the Riemannian foliation of the $3$-dimensional Carri\`ere torus is necessarily nontaut. Therefore, the given bundle-like metric of $T^3_A$ cannot be a solution of transverse vacuum static equation, as $S^Q$ is constant.
\end{eg} 

Recently, another example of a $4$-codimensional nontaut Riemannian foliation on a compact $7$-dimensional manifold has been constructed by Habib, Richardson, and Wolak \cite{HRW}. We compute the symmetric tautness tensor and the adjusted Ricci curvature of this example, to determine that such foliated manifold is a solution of transverse vacuum static space. 

\begin{eg}\normalfont (HRW example)
Let us consider a $7$-dimensional solvable Lie group $G$ of the following $9\times 9$ matrices $A=(a_{ij}(x_1,x_2,x_3,x_4,x_5,x_6,x_7))$ of real entries.
\begin{equation}
\begin{split}
&a_{11}=a_{44}=\dfrac{1}{a_{33}}=\dfrac{1}{a_{66}}=e^{kx_7}\\
&a_{22}=a_{55}=a_{77}=a_{88}=1\\
&a_{21}=-n_1x_2e^{kx_7}\\
&a_{45}=-n_2x_5e^{kx_7}\\
&(a_{i9})_{i=1}^9=(x_1,x_3,x_2,x_4,x_6,x_5,x_7,0,1)^T\\
&\text{ else, }a_{ij}=0 ,
\end{split}
\end{equation} where $n_1, n_2 \neq 0$ are real constants and $k$ is an another real constant satisfying $2\text{cosh}(k) \geq 3$ is an integer.
Then the left-invariant vector fields on $G$ are calculated as follows.
\begin{equation}
\begin{split}
&\mathbf{e}_1=e^{kx_7}\left(\dfrac{\partial}{\partial x_1}-n_1x_2\dfrac{\partial}{\partial x_3}\right)\\
&\mathbf{e}_2=e^{-kx_7}\dfrac{\partial}{\partial x_2}\\
&\mathbf{e}_4=e^{kx_7}\left(\dfrac{\partial}{\partial x_4}-n_2x_5\dfrac{\partial}{\partial x_6}\right)\\
&\mathbf{e}_5=e^{-kx_7}\dfrac{\partial}{\partial x_5}\\
&\mathbf{e}_3=\dfrac{\partial}{\partial x_3},\mathbf{e}_6=\dfrac{\partial}{\partial x_6},\mathbf{e}_7=\dfrac{\partial}{\partial x_7}
\end{split}
\end{equation}
Now we consider the left-invariant Riemannian metric which makes $\{\mathbf{e}_i\}$ an orthonormal basis of $G$ as follows.
\begin{equation}
g=\sum_{i=1}^7\mathbf{e}_i^{\flat} \otimes \mathbf{e}_i^{\flat},
\end{equation}where $\{\mathbf{e}_i^{\flat}\}$ is the dual basis of $\{\mathbf{e}_i\}$.

Additionally, we assume the following Lie bracket structure $[,]$ on $(G,g)$.
\begin{equation}
\begin{split}
&g([\mathbf{e}_1,\mathbf{e}_7],\mathbf{e}_1)=g([\mathbf{e}_4,\mathbf{e}_7],\mathbf{e}_4)=-k\\
&g([\mathbf{e}_2,\mathbf{e}_7],\mathbf{e}_2)=g([\mathbf{e}_5,\mathbf{e}_7],\mathbf{e}_5)=k\\
&g([\mathbf{e}_1,\mathbf{e}_2],\mathbf{e}_3)=n_1\\
&g([\mathbf{e}_4,\mathbf{e}_5],\mathbf{e}_6)=n_2\\
&\text{ else } g([\mathbf{e}_i,\mathbf{e}_j],\mathbf{e}_k)=0,
\end{split}
\end{equation}

Also, let us assume a discrete cocompact action $\Gamma$ on $(G,[,],g)$. For details of the aforementioned action, we refer to \cite{HRW} again.

Applying (6.11), The transverse Riemannian connection of $(G,[,],g)$ is directly derived. In other words, we directly have the following calculations of transverse Riemann curvature of $(G/\Gamma,\F,g)$, where $\F$ is a Riemannian foliation spanned by $\{\mathbf{e}_3,\mathbf{e}_4,\mathbf{e}_6\}$.
\begin{equation}
    \begin{split}
        &R^Q(\mathbf{e}_i,\mathbf{e}_7,\mathbf{e}_i,\mathbf{e}_7)=-k^2, i=1,2,5\\
        &R^Q(\mathbf{e}_i,\mathbf{e}_j,\mathbf{e}_i,\mathbf{e}_j)=0 \text{ else.}
    \end{split}
\end{equation}

Therefore, the transverse scalar curvature $S^Q$ is calculated as follows.
\begin{equation}
    S^Q=\sum_{i,j=1,2,5,7} R^Q (\mathbf{e}_i,\mathbf{e}_j,\mathbf{e}_i,\mathbf{e}_j)=-6k^2.
\end{equation}

On the other hand, the mean curvature $\kappa_b$ is calculated as follows.
\begin{equation}
    \kappa_b=\sum_{i=3,4,6} D_{\mathbf{e}_i}\mathbf{e}_i= D_{\mathbf{e}_4}\mathbf{e}_4 = k \mathbf{e}_7,
\end{equation} which means $\kappa_b$ is basic and $\F$ is nontaut, by (6.11).

Therefore, $\F$ is a nontaut foliation with a constant $S^Q$, which is not a solution of a transverse vacuum static space equation by Theorem 1.1.

\end{eg}

\addresses


\begin{thebibliography}{00}
\bibitem{AL} J. A. Alvarez-L\'opez, \textit{The basic component of the mean curvature of Riemannian foliations,} Ann. Glob. Anal. Geom. 10, 179–194, (1992).
\bibitem{Bes} A. Besse, \textit{Einstein manifolds,} Springer-Verlag, New York, (1987).
\bibitem{Bou} J. P. Bourguignon, \textit{Une stratifcation de l’espace des structures riemanniennes,} Compositio Math. 30, 1-41, (1975).
\bibitem{Car} Y. Carri\`ere, \textit{Flots riemanniens,} Transversal structure of foliations(Toulouse, 1982), 31–52 (1984).
\bibitem{Cha} I. Chavel, \textit{Eigenvalues in Riemannian geometry,}  Academic Press, (1984).
\bibitem{CLN} B. Chow, P. Lu, L. Ni, \textit{Hamilton’s Ricci Flow,} Graduate Studies in Mathematics, vol. 77, p. 608. American Mathematical Society; Science Press, Providence, RI; Beijing, New York (2006).
\bibitem{FM} A. E. Fischer and J. E. Marsden, \textit{Deformations of the scalar curvature,} Duke Math. J. 42(3), 519-547, (1975).
\bibitem{Heb} J. J. Hebda, \textit{An example relevant to curvature pinching theorems for Riemannian foliations,} Proc. Amer. Math. Soc., 114(1), 195-199, (1992).
\bibitem{HJM} S. Hwang, S. D. Jung and J. Moon, \textit{Transverse Ricci solitons on a compact Riemannian foliation,} Rev. R. Acad. Cienc. Exactas Fís. Nat. Ser. A Mat. RACSAM 120(1), 23, (2026).
\bibitem{HR} G. Habib and K. Richardson, \textit{Modified differentials and basic cohomology for Riemannian foliations,} J. Geom. Anal. 23(3), 1314-1342, (2013).
\bibitem{HRW} G. Habib, K. Richardson and R. Wolak, \textit{Transverse geometric formality,} Math. Z., 309(2), 20, (2025).  
\bibitem{HY} S.Hwang and G. Yun, \textit{On the structure of compact static vacuum spaces and positive isotropic curvature,}  Math. Phys. Anal. Geom. 29, 15, (2026).
\bibitem{KW1} J. L. Kazdan and F. W. Warner, \textit{Existence and conformal deformation of metrics with prescribed Gaussian and scalar curvature,} Annals. Math., Second Series, 101(2), 317-331, (1975).
\bibitem{KW2} J. L. Kazdan and F. W. Warner, \textit{Scalar curvature and conformal deformation of Riemannian structure,} J. Differential Geom., 10(1), 113-134, (1975). 
\bibitem{Jun} S.D. Jung, \textit{Eigenvalue estimates for the basic Dirac operator on a Riemannian foliation admitting a basic harmonic 1-form,} J. Geom. Phys. 57(4), 1239–1246, (2007). 
\bibitem{KS} J. Kim and J. Shin, \textit{Four-dimensional static and related critical spaces with harmonic curvature,} Pacific J. Math., 295(2), 429-462, (2018).
\bibitem{Kob} O. Kobayashi, \textit{A differential equation arising from scalar curvature fucntion,} J. Math. Soc. of Japan, 34(4), 665-675, (1982).
\bibitem{Laf} J. Lafontaine, \text{Sur la g\'eom\'etrie d’une g\'en\'eralisation de l’equation differentielle d’Obata,} J. Math. Pures Appl., 62, no. 1, 63-72, (1983).
\bibitem{Lin} D. Lin \textit{Transverse $\F^T$-entropy for Riemannian foliations,} J. Math. Anal. Appl. 556  no. 1, part 1, Paper No. 130070, (2026).
\bibitem{Rum} H. Rummler, \textit{Quelques notions simples en g\'eom\'etrie riemannienne et leurs applications aux feuilletages compacts,} Comment. Math. Helv. 54(2), 224–239 (1979).
\bibitem{Ton} Ph. Tondeur, \textit{Geometry of foliations,} Birkh\"auser Basel (1997).
\bibitem{WZ} G. Wang and Y. Zhang, \textit{A conformal integral invariant on Riemannian foliations,} Proc. Amer. Math. Soc., 141(4), 1405–1414, (2013).

\end{thebibliography}
\end{document}